\documentclass[12pt,reqno]{amsart}
\usepackage{amsmath, amsthm, amssymb}

\advance \topmargin by -\headheight
\advance \topmargin by -\headsep
     
\evensidemargin \oddsidemargin
\newtheorem{theorem}{Theorem}[section]
\newtheorem{lemma}[theorem]{Lemma}
\newtheorem{corollary}[theorem]{Corollary}

\theoremstyle{definition}

\theoremstyle{remark}

\numberwithin{equation}{section}

\newcommand{\mmod}[1]{\,\,(\text{mod}\,\,#1)}

\def\bfh{{\mathbf h}}

\def\bfm{{\mathbf m}}
\def\bfn{{\mathbf n}}

\def\bfz{{\mathbf z}}

\def\calA{{\mathcal A}}  

\def\calB{{\mathcal B}}

\def\calD{{\mathcal D}}

 \def\Ftil{{\widetilde F}}

\def\calU{{\mathcal U}} \def\Util{{\widetilde U}}

\def\ftil{\widetilde{f}}

\def\dbN{{\mathbb N}}

\def\dbZ{{\mathbb Z}}

\def\grB{{\mathfrak B}}

\def\grm{{\mathfrak m}}\def\grM{{\mathfrak M}}\def\grN{{\mathfrak N}}
\def\grn{{\mathfrak n}}
\def\grB{{\mathfrak B}}

\def\grp{{\mathfrak p}} \def\grP{{\mathfrak P}}
\def\grq{{\mathfrak q}} \def\grQ{{\mathfrak Q}}

\def\alp{{\alpha}} 
\def\bet{{\beta}}  
\def\gam{{\gamma}} 
\def\del{{\delta}} \def\Del{{\Delta}}

\def\tet{{\theta}}  
\def\kap{{\kappa}}
 \def\Lam{{\Lambda}}

 \def\Sig{{\Sigma}} 

\def\Ups{{\Upsilon}} 

\def\ome{{\omega}} \def\Ome{{\Omega}}
\def\d{{\partial}}
\def\eps{\varepsilon}

\def\le{\leqslant} \def\ge{\geqslant}

\def\d{{\,{\rm d}}}

\begin{document}
\title[Sums of three cubes]{Sums of three cubes, II}
\author[Trevor D. Wooley]{Trevor D. Wooley}
\address{School of Mathematics, University of Bristol, University Walk, Clifton, Bristol 
BS8 1TW, United 
Kingdom}
\email{matdw@bristol.ac.uk}
\date{}
\dedicatory{In honorem R. C. Vaughan annos LXX nati}
\begin{abstract} Estimates are provided for $s$th moments of cubic smooth Weyl sums, 
when $4\le s\le 8$, by enhancing the author's iterative method that delivers estimates 
beyond classical convexity. As a consequence, an improved lower bound is presented for 
the number of integers not exceeding $X$ that are represented as the sum of three cubes 
of natural numbers.\end{abstract}
\maketitle

\section{Introduction} A heuristic application of the Hardy-Littlewood (circle) method 
suggests that the set of integers represented as the sum of three cubes of natural 
numbers should have positive density. Although intense effort over the past $75$ years 
has delivered a reasonable approximation to this expectation, an unconditional proof 
remains elusive. However, each phase of progress has been accompanied by technological 
advances of value elsewhere in applications of the circle method, and so even modest 
advances remain of interest. The most recent progress \cite{Woo2000b} hinges on an 
extension of Vaughan's method \cite{Vau1989} utilising smooth numbers, in which 
fractional moments of exponential sums are estimated non-trivially. In this paper, we make 
further progress on sums of three cubes by exploiting a new mean value estimate to 
improve earlier estimates for fractional moments of cubic smooth Weyl sums. Although 
these improvements are modest in scale, such estimates have found many applications 
(see, for example, \cite{BBW1995}, \cite{BW2007}, \cite{BW2015a}), and it seems 
reasonable to expect that our new bounds will also be of considerable utility.\par

We begin with a new lower bound for for the number, $N(X)$, of integers not exceeding 
$X$ which are the sum of three cubes of natural numbers.

\begin{theorem}\label{theorem1.1} One has $N(X)\gg X^\bet$, where 
$\bet=0.91709477$.
\end{theorem}

Lower bounds for $N(X)$ are at least implicit in work of Hardy and Littlewood 
\cite{HL1925} from 1925. By developing methods based on diminishing ranges and their 
$p$-adic variants, Davenport \cite{Dav1939} established the lower bound 
$N(X)\gg X^{13/15-\eps}$, subsequently obtaining $N(X)\gg X^{47/54-\eps}$ (see 
\cite{Dav1950}). Thirty-five years later, Vaughan \cite{Vau1985}, \cite{Vau1986} 
enhanced these methods, first proving that $N(X)\gg X^{8/9-\eps}$, and later that 
$N(X)\gg X^{19/21-\eps}$. His seminal introduction \cite{Vau1989} of methods utilising 
smooth numbers led to the lower bound $N(X)\gg X^{11/12-\eps}$ (see also Ringrose 
\cite{Rin1986} for an intermediate result). The author's derivation of effective estimates 
for fractional moments of smooth Weyl sums \cite{Woo1995} first delivered a lower 
bound of the shape $N(X)\gg X^{1-\xi/3-\eps}$, where $\xi=0.24956813\ldots$ denotes 
the positive root of the polynomial $\xi^3+16\xi^2+28\xi-8$. Subsequently, the author 
obtained a similar estimate in which $\xi=(\sqrt{2833}-43)/41=0.24941301\ldots$ 
(see \cite{Woo2000b}). With this value of $\xi$, one has $1-\xi/3=0.91686232\ldots $, 
which should be compared with the exponent $0.91709477$ of Theorem 
\ref{theorem1.1}. Subject to the truth of an unproved Riemann Hypothesis concerning 
a certain Hasse-Weil $L$-function, meanwhile, one has the conditional estimate 
$N(X)\gg X^{1-\eps}$ due to Hooley \cite{Hoo1986,Hoo1997} and Heath-Brown 
\cite{HB1998}.\par

Theorem \ref{theorem1.1} follows from an estimate for the sixth moment of a certain 
smooth Weyl sum. Define the set of $R$-smooth numbers of size at most $P$ by
$$\calA(P,R)=\{ n\in [1,P]\cap \dbZ: \text{$p|n$ and $p$ prime}\Rightarrow p\le R\}.$$
Then, with $e(z)=e^{2\pi iz}$, we introduce the smooth and classical Weyl sums
\begin{equation}\label{1.1}
f(\alp;P,R)=\sum_{x\in \calA(P,R)}e(\alp x^3)\quad \text{and}\quad 
F(\alp;P)=\sum_{1\le x\le P}e(\alp x^3).
\end{equation}
In \S7 we establish the mean value estimate contained in the following theorem.

\begin{theorem}\label{theorem1.2} Write $\del_6=0.24871567$. Then there exists a 
positive number $\eta$ with the property that, whenever $R\le P^\eta$, one has
\begin{equation}\label{1.2}
\int_0^1|F(\alp;P)^2f(\alp;P,R)^4|\d\alp \ll P^{3+\del_6}.
\end{equation}
\end{theorem}

For comparison, \cite[Theorem 1.2]{Woo2000b} yields a similar estimate with 
$\del_6=0.24941301\ldots $, whilst the earlier work of Vaughan \cite{Vau1989} 
provides an analogous sixth moment estimate for $f(\alp;P,R)$ with associated exponent 
$\del_6=\tfrac{1}{4}+\eps$, for any $\eps>0$. Note that in many applications 
(see \cite{BW2007, BW2015a, BW2015b}), it is crucial that (\ref{1.2}) hold with 
$\del_6<\tfrac{1}{4}$, hence the significance of Theorem \ref{theorem1.2}.\par

The bound (\ref{1.2}) of Theorem \ref{theorem1.2} leads to improvement in estimates 
associated with the unrepresentation theory of Waring's problem for cubes. Let $E_s(X)$ 
denote the number of integers not exceeding $X$ which are {\it not} the sum 
of $s$ cubes of natural numbers. Then the arguments of Br\"udern \cite{Bru1991} 
and Kawada and Wooley \cite{KW2010} lead to the estimates recorded in the following 
theorem.

\begin{theorem}\label{theorem1.3} Write 
$\tau=\tfrac{2}{7}\left( \tfrac{1}{4}-0.24871567\right) =1/2725.15\dots $. Then 
one has
$$E_4(X)\ll X^{37/42-\tau},\quad E_5(X)\ll X^{5/7-\tau},\quad 
E_6(X)\ll X^{3/7-2\tau}.$$
\end{theorem}

The aforementioned work of Br\"udern \cite{Bru1991} yields the bound 
$E_4(X)\ll X^{37/42+\eps}$, whilst Kawada and Wooley \cite[Theorem 1.4]{KW2010} 
obtain a conclusion similar to that of Theorem \ref{theorem1.3}, though with $\tau$ 
slightly smaller than $1/5962$. We will not discuss the (routine) proof of Theorem 
\ref{theorem1.3} further here, noting merely that the conclusion of Theorem 
\ref{theorem1.2} is the key input into the methods of \cite{Bru1991}.\par

We establish Theorem \ref{theorem1.2} as a consequence of estimates for the mean 
values
\begin{equation}\label{1.3}
U_s(P,R)=\int_0^1|f(\alp;P,R)|^s\d\alp ,
\end{equation}
with $4\le s\le 8$. The iterative method of \cite{Woo1995} obtains a bound for $U_s(P,R)$ 
in terms of corresponding bounds for $U_{s-2}(P,R)$ and $U_t(P,R)$, wherein $t$ is a 
parameter to be chosen with $\tfrac{4}{3}(s-2)\le t\le 2(s-2)$. A key player in 
determining the strength of these estimates is an exponential sum of the shape
$$\Ftil_1(\alp)=\sum_{\substack{u\in \calA(P^\tet R,R)\\ u>P^\tet}}
\sum_{\substack{z_1,z_2\in \calA(P,R)\\ z_1\equiv z_2\mmod{u^3}\\ z_1\ne z_2}}e(\alp 
u^{-3}(z_1^3-z_2^3)),$$
in which $\tet$ is a parameter with $0\le \tet\le \tfrac{1}{3}$. This exponential sum is 
made awkward to handle by the constraint that the summands $z_1$ and $z_2$ be 
smooth. In this paper we estimate the auxiliary integral
$$\int_0^1\Ftil_1(\alp)|f(\alp;P^{1-\tet},R)|^{s-2}\d\alp $$
in terms of the mediating mean value
$$\int_0^1|\Ftil_1(\alp)^2f(\alp;P^{1-\tet},R)^2|\d\alp .$$
By orthogonality, the latter counts the number of solutions of an underlying Diophantine 
equation. By discarding the smoothness constraint implicit in the sum $\Ftil_1(\alp)$, much 
of the strength of the Hardy-Littlewood method may be preserved in the ensuing minor 
arc estimate. After preparing an auxiliary estimate in \S2, we analyse this new mean value 
in \S3, and indicate in \S4 how it may be utilised in the method of \cite{Woo1995}. Ideas 
relevant for the estimation of the mean value $U_s(P,R)$ when $s=6$, and when $s>6.5$, 
are presented in \S5.\par

The Keil-Zhao device (see \cite[page 608]{Kei2014} and the discussion leading to 
\cite[equation (3.10)]{Zha2014}) enables us in \S6 to obtain stronger minor arc estimates 
for smooth Weyl sums than available hitherto. When $\grm\subseteq [0,1)$, $0<t\le 2$ and 
$s\ge 6$, this idea delivers an estimate of the shape
$$\int_\grm |f(\alp;P,R)|^{s+t}\d\alp \ll P^{t/2}\Bigl( \sup_{\alp \in \grm}|F(\alp;P)|
\Bigr)^{t/2}\int_0^1|f(\alp;P,R)|^s\d\alp ,$$
in place of
$$\int_\grm |f(\alp;P,R)|^{s+t}\d\alp \ll \Bigl( \sup_{\alp \in \grm}|f(\alp;P,R)|\Bigr)^t
\int_0^1|f(\alp;P,R)|^s\d\alp .$$
The ease with which classical Weyl sums can be estimated on sets of minor arcs ensures 
that this device is of utility when $s$ lies between $6$ and $8$. In particular, in \S7 we 
explain how to improve \cite[Theorem 2]{BW2001}, which establishes that when $R$ is a 
small enough power of $P$, then $U_s(P,R)\ll P^{s-3}$ for $s\ge 7.691$.

\begin{theorem}\label{theorem1.4} Suppose that $\eta>0$ and $P$ is sufficiently large in 
terms of $\eta$, and further that $R\le P^\eta$. Then provided that $s\ge 7.5906$, one 
has
$$\int_0^1|f(\alp;P,R)|^s\d\alp \ll P^{s-3}.$$
\end{theorem}

Our estimates for the mean values $U_s(P,R)$ depend on those for $U_t(P,R)$ for 
appropriate choices of $t$. In \S7, we describe how computations associated with this 
complicated iteration were performed, and discuss the extent to which the computed 
exponents reflect the sharpest available from this circle of ideas. These conclusions are 
summarised in the following theorem.

\begin{theorem}\label{theorem1.5}
Let $(s,\del_s,\Del_s)$ be a triple listed in Table 1. Suppose that $\eta>0$ and $P$ is 
sufficiently large in terms of $\eta$, and further that $R\le P^\eta$. Then
$$\int_0^1|f(\alp;P,R)|^s\d\alp \ll P^{s/2+\del_s}\quad \text{and}\quad 
\int_0^1|f(\alp;P,R)|^s\d\alp \ll P^{s-3+\Del_s}.$$
\end{theorem}

Exponents may be derived for values of $s$ between those in the table by linear 
interpolation using H\"older's inequality. Values of $\del_s$ and $\Del_s$ computed in \S7 
have been rounded up, as appropriate, in the final decimal place recorded.\par

\begin{table} 
    \begin{tabular}{ | l | l | l || l | l | l | p{5cm} |}
    \hline
    \ $s$ & \ \ \ \ \ \ $\del_s$ & \ \ \ \ \ $\Del_s$ & \ $s$ & \ \ \ \ \ \ $\del_s$ & 
\ \ \ \ \ $\Del_s$ \\ \hline
    4.0&0.00000000&1.00000000&6.0&0.24871567&0.24871567\\ \hline
    4.1&0.00130000&0.95130000&6.1&0.27667792&0.22667792\\ \hline
    4.2&0.00495852&0.90495852&6.2&0.30598066&0.20598066\\ \hline
    4.3&0.01069296&0.86069296&6.3&0.33718632&0.18718632\\ \hline
    4.4&0.01811263&0.81811263&6.4&0.36984515&0.16984515\\ \hline
    4.5&0.02685074&0.77685074&6.5&0.40263501&0.15263501\\ \hline
    4.6&0.03754195&0.73754195&6.6&0.43542486&0.13542486\\ \hline
    4.7&0.04903470&0.69903470&6.7&0.46851012&0.11851012\\ \hline
    4.8&0.06130069&0.66130069&6.8&0.50330866&0.10330866\\ \hline
    4.9&0.07426685&0.62426685&6.9&0.53863866&0.08863866\\ \hline
    5.0&0.08780854&0.58780854&7.0&0.57423853&0.07423853\\ \hline
    5.1&0.10328796&0.55328796&7.1&0.61131437&0.06131437\\ \hline
    5.2&0.11894874&0.51894874&7.2&0.64881437&0.04881437\\ \hline
    5.3&0.13477800&0.48477800&7.3&0.68631437&0.03631437\\ \hline
    5.4&0.15076406&0.45076406&7.4&0.72381437&0.02381437\\ \hline
    5.5&0.16689626&0.41689626&7.5&0.76131437&0.01131437\\ \hline
    5.6&0.18316493&0.38316493&7.6&0.80000000&0.00000000\\ \hline
    5.7&0.19954296&0.34954296&7.7&0.85000000&0.00000000\\ \hline 
    5.8&0.21593386&0.31593386&7.8&0.90000000&0.00000000\\ \hline
    5.9&0.23232477&0.28232477&7.9&0.95000000&0.00000000\\ \hline    
\end{tabular}
\vskip.2cm
\caption{Associated and permissible exponents for $4\le s\le 8$.}
\end{table}

In this paper, we adopt the convention that whenever $\eps$, $P$ or $R$ appear in a 
statement, either implicitly or explicitly, then for each $\eps>0$, there exists a 
positive number $\eta=\eta(\eps)$ such that the statement holds whenever $R\le P^\eta$ 
and $P$ is sufficiently large in terms of $\eps$ and $\eta$. Implicit constants in 
Vinogradov's notation $\ll$ and $\gg$ will depend at most on $\eps$ and $\eta$. Since our 
iterative methods involve only a finite number of statements (depending at most on 
$\eps$), there is no danger of losing control of implicit constants. Finally, write 
$\|\tet\|=\underset{y\in\dbZ}{\min}|\tet-y|$.

\section{An auxiliary mean value estimate} Before announcing our pivotal mean value 
estimate, we introduce some notation. Let $\phi$ be a real number with 
$0\le \phi\le \tfrac{1}{3}$, and write
\begin{equation}\label{2.1}
M=P^\phi,\quad H=PM^{-3}\quad \text{and}\quad Q=PM^{-1}.
\end{equation}
Define the exponential sums
\begin{equation}\label{2.2}
F_1(\alp)=\sum_{1\le z\le 2P}\sum_{1\le h\le H}\sum_{M<m\le MR}
e(2\alp h(3z^2+h^2m^6)),
\end{equation}
$$D(\alp)=\sum_{1\le h\le H}\biggl| \sum_{1\le z\le 2P}e(6\alp hz^2)\biggr|^2$$
and
\begin{equation}\label{2.3}
E(\alp)=\sum_{1\le h\le H}\biggl| \sum_{M<m\le MR}e(2\alp h^3m^6)\biggr|^2.
\end{equation}
Also, when $\grB\subseteq [0,1)$, we introduce the mean value
\begin{equation}\label{2.4}
\Ups(P,R;\phi;\grB)=\int_\grB|F_1(\alp)^2f(\alp;2Q,R)^2|\d\alp ,
\end{equation}
and then write $\Ups(P,R;\phi)=\Ups(P,R;\phi;[0,1))$. We observe that an application of 
Cauchy's inequality to (\ref{2.2}) yields the bound $|F_1(\alp)|^2\le D(\alp)E(\alp)$. 
Consequently, when $t\ge 2$, we obtain the estimate
\begin{equation}\label{2.5}
\Ups(P,R;\phi;\grB )\le \int_\grB \left( D(\alp)E(\alp)\right)^{2/t}|F_1(\alp)|^{2-4/t}
|f(\alp;2Q,R)|^2\d\alp .
\end{equation}

Recall the definition (\ref{1.3}) of the mean value $U_s(P,R)$. We say that an exponent 
$\mu_s$ is {\it permissible} whenever it has the property that, with the notational 
conventions introduced above, one has $U_s(P,R)\ll P^{\mu_s+\eps}$. It follows that, for 
each  positive number $s$, a permissible exponent $\mu_s$ exists with 
$s/2\le \mu_s\le s$. We refer to the exponent $\del_s$ as {\it associated} when 
$\mu_s=s/2+\del_s$ is permissible, and $\Del_s$ as {\it admissible} when 
$\mu_s=s-3+\Del_s$ is permissible.\par

We require a Hardy-Littlewood dissection. Let $\grm$ denote the set of points 
$\alp\in [0,1)$ with the property that, whenever there exist $a\in \dbZ$ and $q\in \dbN$ 
with $(a,q)=1$ and $|q\alp-a|\le PQ^{-3}$, then one has $q>P$. Further, let 
$\grM=[0,1)\setminus \grm$.

\begin{lemma}\label{lemma2.1} Suppose that $t\ge 4$ and $0\le \phi\le \tfrac{1}{3}$. 
Then whenever $\del_t$ is an associated exponent, one has
$$\Ups(P,R;\phi;\grm)\ll P^{1+\eps}MH^{1+2/t}Q^{1+2\del_t/t}.$$
\end{lemma}

\begin{proof} We ultimately work outside the range $0\le \phi\le \tfrac{1}{7}$ in which 
the estimate
$$\sup_{\alp \in \grm}|F_1(\alp)|\ll P^\eps (PM)^{1/2}H$$
follows from \cite[Lemmata 3.1 and 3.4]{Vau1989}, and so we engineer a hybrid method 
combining elements of the Hardy-Littlewood method with a Diophantine interpretation of 
auxiliary equations. We begin by applying H\"older's inequality to (\ref{2.5}), obtaining 
the bound
\begin{equation}\label{2.6}
\Ups(P,R;\phi;\grm )\le \Bigl( \sup_{\alp\in \grm}D(\alp)\Bigr)^{2/t}I_1^{2/t}I_2^{1-4/t}
U_t(2Q,R)^{2/t},
\end{equation}
where $U_t(2Q,R)$ is defined via (\ref{1.3}),
\begin{equation}\label{2.7}
I_1=\int_0^1E(\alp)|F_1(\alp)|^2\d\alp \quad \text{and}\quad 
I_2=\int_0^1|F_1(\alp)|^2\d\alp .
\end{equation}

\par The estimates
\begin{equation}\label{2.8}
I_2\ll P^{1+\eps}MH\quad \text{and}\quad U_t(2Q,R)\ll Q^{t/2+\del_t+\eps}
\end{equation}
follow, respectively, from \cite[Lemma 2.3]{Vau1989} with $j=1$ and the definition of an 
associated exponent. Also, given $\alp\in [0,1)$, we find from \cite[Lemma 3.1]{Vau1989} 
that whenever $a\in \dbZ$ and $q\in \dbN$ satisfy $(a,q)=1$ and $|\alp-a/q|\le q^{-2}$, 
then
\begin{equation}\label{2.9}
D(\alp)\ll P^\eps \biggl( \frac{P^2H}{q+Q^3|q\alp-a|}+PH+q+Q^3|q\alp-a|\biggr) .
\end{equation}
By Dirichlet's theorem on Diophantine approximation, there exist $a\in \dbZ$ and 
$q\in \dbN$ with $0\le a\le q\le P^{-1}Q^3$, $(a,q)=1$ and $|q\alp -a|\le PQ^{-3}$. 
When $\alp\in \grm$, it follows that $q>P$, and hence we deduce via (\ref{2.1}) that
\begin{equation}\label{2.10}
\sup_{\alp\in \grm}D(\alp)\ll P^\eps (PH+P^{-1}Q^3)\ll P^{1+\eps}H.
\end{equation}

\par Finally, by reference to (\ref{2.2}), (\ref{2.3}) and (\ref{2.7}), it follows from 
orthogonality that $I_1$ counts the number of integral solutions of the equation
\begin{equation}\label{2.11}
h_0^3(n_1^6-n_2^6)=h_1(3z_1^2+h_1^2m_1^6)-h_2(3z_2^2+h_2^2m_2^6),
\end{equation}
with
$$1\le h_0,h_1,h_2\le H,\quad M<n_1,n_2,m_1,m_2\le MR\quad \text{and}\quad 1\le 
z_1,z_2\le 2P.$$
Let $N_1$ denote the number of solutions of (\ref{2.11}) counted by $I_1$ in which 
$n_1=n_2$, let $N_2$ denote the corresponding number in which 
$h_1z_1^2\ne h_2z_2^2$, and let $N_3$ denote the number with $n_1\ne n_2$ and 
$h_1z_1^2=h_2z_2^2$. Thus $I_1\le N_1+N_2+N_3$.\par

By orthogonality, it follows from (\ref{2.2}) and (\ref{2.11}) with $n_1=n_2$ that
$$N_1\le HMR\int_0^1|F_1(\alp)|^2\d\alp ,$$
and hence we deduce from (\ref{2.7}) and (\ref{2.8}) that
\begin{equation}\label{2.12}
N_1\ll P^{1+\eps}M^2H^2.
\end{equation}

\par When $\bfh,\bfm,\bfn,\bfz$ is a solution of (\ref{2.11}) counted by $N_2$, the 
integer
$$L=h_0^3(n_1^6-n_2^6)-h_1^3m_1^6+h_2^3m_2^6$$
is non-zero. There are $O(H^3(MR)^4)$ possible choices for $L$, and we find from 
(\ref{2.11}) that for each fixed choice one has $3(h_1z_1^2-h_2z_2^2)=L$. With $h_1$ 
and $h_2$ already fixed, it follows from \cite[Lemma 3.5]{VW1995} that the number of 
possible choices for $z_1$ and $z_2$ is $O((h_1h_2|L|P)^\eps)$. Thus we 
conclude that
\begin{equation}\label{2.13}
N_2\ll P^\eps H^3M^4.
\end{equation}

\par Finally, consider a solution $\bfh,\bfm,\bfn,\bfz$ counted by $N_3$. Given $h_2$ and 
$z_2$, an elementary estimate for the divisor function shows that the number of possible 
choices for $h_1$ and $z_1$ satisfying $h_1z_1^2=h_2z_2^2$ is $O((HP)^\eps)$. Fix 
any one amongst these $O((HP)^{1+\eps})$ possible choices for $h_1,h_2,z_1,z_2$. One 
finds from (\ref{2.11}) that $h_0,\bfm,\bfn$ satisfy the equation
$$(h_1m_1^2)^3-(h_2m_2^2)^3=h_0^3(n_1^6-n_2^6).$$
Since $n_1\ne n_2$, the right hand side here is non-zero, and hence also the left hand 
side. Thus, again applying a divisor function estimate, it follows that for any one amongst 
the $O((MR)^2)$ possible choices for $m_1$ and $m_2$, there are $O(P^\eps)$ possible 
choices for $h_0$, $n_1-n_2$ and $n_1^5+n_1^4n_2+\ldots +n_2^5$. We deduce that 
there are just $O(P^\eps)$ possible choices for $h_0$, $n_1$ and $n_2$, and thus
\begin{equation}\label{2.14}
N_3\ll P^\eps (HP)^{1+\eps}(MR)^2\ll P^{1+3\eps}HM^2.
\end{equation}

\par On combining (\ref{2.12})--(\ref{2.14}), we conclude via (2.1) that
$$I_1=N_1+N_2+N_3\ll P^\eps(PM^2H^2+H^3M^4)\ll P^{1+\eps}M^2H^2.$$
Substituting this estimate together with (\ref{2.8}) and (\ref{2.10}) into (\ref{2.6}), we 
arrive at the upper bound
$$\Ups(P,R;\phi;\grm)\ll P^\eps (PH)^{2/t}(PM^2H^2)^{2/t}(PMH)^{1-4/t}
Q^{1+2\del_t/t},$$
and the conclusion of the lemma follows with a modicum of computation.
\end{proof}

We require a complementary major arc estimate.

\begin{lemma}\label{lemma2.2} Suppose that $t\ge 4$ and $0\le \phi\le \tfrac{1}{3}$. 
Then whenever $\del_t$ is an associated exponent, one has
$$\Ups(P,R;\phi;\grM)\ll P^{1+\eps}MH^{1+2/t}Q^{1+2\del_t/t}.$$
\end{lemma}

\begin{proof} The major arcs $\grM$ are contained in the union of the intervals
$$\grM(q,a)=\{ \alp\in [0,1):|q\alp-a|\le PQ^{-3}\},$$
with $0\le a\le q\le P$ and $(a,q)=1$. Define $\Del(\alp)$ for $\alp\in [0,1)$ by putting
$$\Del(\alp)=(q+Q^3|q\alp-a|)^{-1},$$
when $\alp\in \grM(q,a)\subseteq \grM$, and otherwise by setting $\Del(\alp)=0$. Then it 
follows 
from (\ref{2.9}) that when $\alp\in \grM$, one has
\begin{equation}\label{2.15}
D(\alp)\ll P^{2+\eps}H\Del(\alp)+P^{1+\eps}H.
\end{equation}
We apply H\"older's inequality to (\ref{2.5}), just as in the treatment of 
$\Ups(P,R;\phi;\grm)$ in the proof of Lemma \ref{lemma2.1}. Thus, by comparing 
(\ref{2.10}) and (\ref{2.15}), we obtain
$$\Ups(P,R;\phi;\grM)\ll P^\eps \left(PMH^{1+2/t}Q^{1+2\del_t/t}+(P^2HT)^{2/t}
\Ups(P,R;\phi;\grM)^{1-2/t}\right),$$
where
\begin{equation}\label{2.16}
T=\int_\grM\Del(\alp)E(\alp)|f(\alp;2Q,R)|^2\d\alp .
\end{equation}
Thus we infer that
\begin{equation}\label{2.17}
\Ups(P,R;\phi;\grM)\ll P^{1+\eps}MH^{1+2/t}Q^{1+2\del_t/t}+P^{2+\eps}HT.
\end{equation}

\par In preparation for the estimation of $T$, we consider the mean value
$$T_0=\int_0^1E(\alp)|f(\alp;2Q,R)|^2\d\alp .$$
By reference to (\ref{2.3}), it follows from orthogonality that $I_6$ counts the number of 
integral solutions of the equation
$$2h^3(n_1^6-n_2^6)=x_1^3-x_2^3,$$
with $1\le h\le H$, $M<n_1,n_2\le MR$ and $x_1,x_2\in \calA(2Q,R)$. Here, the number 
of diagonal solutions with $x_1=x_2$ and $n_1=n_2$ is $O(HMRQ)$. There are 
$O(H(MR)^2)$ possible choices for $h$, $n_1$ and $n_2$ with 
$2h^3(n_1^6-n_2^6)\ne 0$. For each fixed such choice, an elementary estimate for the 
divisor function shows that there are $O(Q^\eps)$ possible choices for $x_1-x_2$ and 
$x_1^2+x_1x_2+x_2^2$, hence also for $x_1$ and $x_2$. Then we conclude via 
(\ref{2.1}) that
\begin{equation}\label{2.18}
T_0\ll P^\eps (HMQ+HM^2)\ll P^{1+\eps}H.
\end{equation}

\par On recalling (\ref{2.3}), one finds that
$$E(\alp)|f(\alp;2Q,R)|^2=\sum_{l\in \dbZ}\psi(l)e(l\alp),$$
where $\psi(l)$ denotes the number of solutions of the equation
$$2h^3(n_1^6-n_2^6)+x_1^3-x_2^3=l,$$
with $1\le h\le H$, $M<n_1,n_2\le MR$ and $x_1,x_2\in \calA(2Q,R)$. In view of 
(\ref{2.18}), one has $\psi(0)=T_0\ll P^{1+\eps}H$. Moreover,
$$\sum_{l\in \dbZ}\psi(l)=E(0)f(0;2Q,R)^2\ll H(MR)^2Q^2.$$
Then by applying \cite[Lemma 2]{Bru1988} within (\ref{2.16}), we deduce via (\ref{2.1}) 
that
$$T\ll Q^{\eps-3}\left( P(P^{1+\eps}H)+H(MR)^2Q^2\right) \ll P^{2\eps}.$$
On substituting this estimate into (\ref{2.17}), we conclude that
$$\Ups(P,R;\phi;\grM)\ll P^{1+\eps}MH^{1+2/t}Q^{1+2\del_t/t}+P^{2+\eps}H.$$
The proof of the lemma is completed by noting that the relations (\ref{2.1}) ensure that 
the second term on the right hand side here is majorised by the first.
\end{proof}

We finish this section by combining the conclusions of Lemmata \ref{lemma2.1} and 
\ref{lemma2.2}.

\begin{lemma}\label{lemma2.3} Suppose that $t\ge 4$ and $0\le \phi\le \tfrac{1}{3}$. 
Then whenever $\del_t$ is an associated exponent, one has
$$\int_0^1|F_1(\alp)^2f(2\alp;Q,R)^2|\d\alp \ll P^{1+\eps}MH^{1+2/t}Q^{1+2\del_t/t}.
$$
\end{lemma}

\begin{proof} On recalling (\ref{2.4}), the desired conclusion follows from Lemmata 
\ref{lemma2.1} and \ref{lemma2.2} by means of the relation
$$\Ups(P,R;\phi)=\Ups(P,R;\phi;\grM)+\Ups(P,R;\phi;\grm).$$
\end{proof}

\section{Further auxiliary mean value estimates} We now introduce notation more closely 
aligned with the author's work \cite{Woo1995, Woo2000a, Woo2000b} on fractional 
moments of smooth Weyl sums. We define the modified set of smooth numbers 
$\calB(L,\pi,R)$ for prime numbers $\pi$ by putting
$$\calB(L,\pi,R)=\{ n\in \calA(L\pi,R):\text{$n>L$ and $\pi|n$}\}.$$
Recall the notation (\ref{2.1}). We define the exponential sums
\begin{align}
\Ftil_{d,e}(\alp;\pi)&=\sum_{u\in \calB(M/d,\pi,R)}\, 
\sum_{\substack{x,y\in \calA(P/(de),R)\\ (x,u)=(y,u)=1\\ x\equiv y\mmod{u^3}\\ 
y<x}}e\left(\alp u^{-3}(x^3-y^3)\right),\label{3.1}\\
F_{d,e}(\alp)&=\sum_{1\le z\le 2P/(de)}\sum_{1\le h\le Hd^2/e}\sum_{M/d<u\le MR/d}
e\left(2\alp h(3z^2+h^2m^6)\right)\label{3.2}
\end{align}
and
\begin{equation}\label{3.3}
\ftil(\alp;P,M,R)=\max_{m>M}\biggl| \sum_{x\in \calA(P/m,R)}e(\alp x^3)\biggr|.
\end{equation}
Note here that $F_{d,e}(\alp)=0$ whenever $e>Hd^2$. Finally, we put
\begin{equation}\label{3.4}
\Ups_{d,e,\pi}(P,R;\phi)=\int_0^1 |\Ftil_{d,e}(\alp;\pi)^2\ftil(\alp;P/(de),M/d,\pi)^2|
\d\alp .
\end{equation}

\par We begin by demystifying the mean value $\Ups_{d,e,\pi}(P,R;\phi)$.

\begin{lemma}\label{lemma3.1}
When $\pi\le R$, one has
$$\Ups_{d,e,\pi}(P,R;\phi)\ll P^\eps \int_0^1|F_{d,e}(\alp)^2f(\alp;2Q/e,R)^2|\d\alp .$$
\end{lemma}

\begin{proof} We first eliminate the maximal aspect of the sum 
$\ftil(\alp;P/(de),M/d,\pi)$ implicit in $\Ups_{d,e,\pi}(P,R;\phi)$. Define
$$\calD_K(\tet)=\sum_{|m|\le K^3}e(m\tet)\quad \text{and}\quad \calD_K^*(\tet)=\min 
\{ 2K^3+1,\|\tet\|^{-1}\},$$
and note that for $K\ge 1$, one has
\begin{equation}\label{3.5}
\int_0^1\calD_K^*(\tet)\d\tet\ll \log (2K).
\end{equation}
On recalling (\ref{2.1}), we find that whenever $m>M$, then one has
$$\sum_{x\in \calA(P/m,R)}e(\alp x^3)=\int_0^1f(\alp+\tet;Q,R)
\calD_{P/m}(\tet)\d\tet.$$
Since $\calD_{P/m}(\tet)\ll \calD_{P/m}^*(\tet)\le \calD_Q^*(\tet)$ for $m>M$, we thus 
infer from (\ref{3.3}) that 
\begin{equation}
\ftil(\alp;P/(de),M/d,\pi)\ll \int_0^1|f(\alp+\tet;Q/e,\pi)|\calD_Q^*(\tet)\d\tet .\label{3.6}
\end{equation}

\par On substituting (\ref{3.6}) into (\ref{3.4}), we deduce that
$$\Ups_{d,e,\pi}(P,R;\phi)\ll \int_{[0,1)^3}|\Ftil_{d,e}(\alp;\pi)^2f_{\tet_1}(\alp)
f_{\tet_2}(\alp)|\calD_Q^*(\tet_1)\calD_Q^*(\tet_2)\d\tet_1\d\tet_2\d\alp ,$$
where, temporarily, we abbreviate $f(\alp+\tet;Q/e,\pi)$ to $f_\tet(\alp)$. Write
\begin{equation}\label{3.7}
\Xi_{d,e,\pi}(\tet)=\int_0^1|\Ftil_{d,e}(\alp;\pi)^2f(\alp+\tet;Q/e,\pi)^2|\d\alp .
\end{equation}
Then by applying the inequality $|z_1z_2|\le |z_1|^2+|z_2|^2$ and invoking symmetry, 
we infer via (\ref{3.5}) that
\begin{align}
\Ups_{d,e,\pi}(P,R;\phi)&\ll \int_0^1\Xi_{d,e,\pi}(\tet_1)\calD_Q^*(\tet_1)\d\tet_1 
\int_0^1\calD_Q^*(\tet_2)\d\tet_2 \notag\\
&\ll Q^\eps \int_0^1 \Xi_{d,e,\pi}(\tet_1)\calD_Q^*(\tet_1)\d\tet_1.\label{3.8}
\end{align}

\par Consider next the integral solutions of the equation
\begin{equation}\label{3.9}
u_1^{-3}(x_1^3-y_1^3)-u_2^{-3}(x_2^3-y_2^3)=w_1^3-w_2^3,
\end{equation}
with, for $i=1$ and $2$, the constraints
$$w_i\in \calA(Q/e,\pi),\quad u_i\in \calB(M/d,\pi,R),\quad x_i,y_i\in \calA(P/(de),R),$$
$$(x_i,u_i)=(y_i,u_i)=1,\quad x_i\equiv y_i\mmod{u_i^3}\quad \text{and}\quad 
y_i<x_i.$$
Then by orthogonality, it follows from (\ref{3.1}) and (\ref{3.7}) that the mean value 
$\Xi_{d,e,\pi}(\tet)$ counts the number of such solutions, with each solution counted with 
weight $e(\tet (w_2^3-w_1^3))$. The latter weight being unimodular, it follows that 
$|\Xi_{d,e,\pi}(\tet)|$ is bounded above by the corresponding number of unweighted 
solutions, and hence by the number of integral solutions of the equation (\ref{3.9}) with, 
for $i=1$ and $2$, the constraints
$$w_i\in \calA(Q/e,R),\quad M/d<u_i\le MR/d,$$
$$1\le y_i<x_i\le P/(de)\quad \text{and}\quad x_i\equiv y_i\mmod{u_i^3}.$$
We now substitute $z_i=x_i+y_i$ and $h_i=(x_i-y_i)u_i^{-3}$ $(i=1,2)$ into equation 
(\ref{3.9}). It follows that $1\le h_i\le (P/(de))(M/d)^{-3}$ for $i=1$ and $2$. Moreover, 
we have $2x_i=z_i+h_iu_i^3$ and $2y_i=z_i-h_iu_i^3$ $(i=1,2)$. Then on noting that
$$u^{-3}\left( (z+hu^3)^3-(z-hu^3)^3\right)=2h(3z^2+h^2u^6),$$
and recalling (\ref{2.1}), we see that $|\Xi_{d,e,\pi}(\tet)|$ is bounded above by the 
number of integral solutions of the equation
$$2h_1(3z_1^2+h_1^2u_1^6)-2h_2(3z_2^2+h_2^2u_2^6)=w_1^3-w_2^3,$$
with, for $i=1$ and $2$,
$$w_i\in \calA(2Q/e,R),\quad M/d<u_i\le MR/d,$$
$$1\le z_i\le 2P/(de)\quad \text{and}\quad 1\le h_i\le Hd^2/e.$$
Then on recalling (\ref{3.2}), it follows by orthogonality that
$$|\Xi_{d,e,\pi}(\tet)|\le \int_0^1|F_{d,e}(\alp)^2f(\alp;2Q/e,R)^2|\d\alp .$$
On substituting this estimate into (\ref{3.8}), we conclude that
$$\Ups_{d,e,\pi}(P,R;\phi)\ll Q^\eps \biggl( \int_0^1 D_Q^*(\tet)\d\tet\biggr) 
\int_0^1|F_{d,e}(\alp)^2f(\alp;2Q/e,R)^2|\d\alp .$$
The conclusion of the lemma now follows on applying the bound (\ref{3.5}).
\end{proof}

\begin{lemma}\label{lemma3.2}
Suppose that
$$1\le d\le M,\quad 1\le e\le \min\{ Q,Hd^2\}\quad \text{and}\quad 0\le \phi\le 
\tfrac{1}{3}.$$
Then, whenever $t\ge 4$ and $\del_t$ is an associated exponent, one has
$$\Ups_{d,e,\pi}(P,R;\phi)\ll d^{4/t}e^{-3-2/t}P^{1+\eps}MH^{1+2/t}Q^{1+2\del_t/t}.
$$
\end{lemma}

\begin{proof} A comparison of (\ref{2.2}) and (\ref{3.2}) reveals that, as a consequence 
of Lemma \ref{lemma2.3} in combination with (\ref{2.1}), whenever $t\ge 4$ and 
$M^3\le P$, one has
\begin{equation}\label{3.10}
\int_0^1|F_{1,1}(\alp)^2f(2\alp;Q,R)^2|\d\alp \ll 
P^{1+\eps}MH^{1+2/t}Q^{1+2\del_t/t}.
\end{equation}
We apply this conclusion with $P/(de)$ in place of $P$ and $M/d$ in place of $M$. In view 
of the relations (\ref{2.1}), we have also $Hd^2/e$ in place of $H$ and $Q/e$ in place of 
$Q$. The hypotheses of the lemma concerning $e$ and $\phi$ then ensure that
$$(M/d)^3(P/(de))^{-1}=e/(Hd^2)\le 1,$$
whence $(M/d)^3\le P/(de)$, confirming the validity of the estimate (\ref{3.10}) with 
these substitutions. Hence we obtain the estimate
\begin{align*}
\int_0^1|F_{d,e}(\alp)^2f(\alp;2Q/e,R)^2|\d\alp &\ll 
\left( \frac{P}{de}\right)^{1+\eps}\left(\frac{M}{d}\right) 
\left( \frac{Hd^2}{e}\right)^{1+2/t}\left( \frac{Q}{e}\right)^{1+2\del_t/t}\\
&\ll d^{4/t}e^{-3-2/t-2\del_t/t}P^{1+\eps}MH^{1+2/t}Q^{1+2\del_t/t}.
\end{align*}
Since Lemma \ref{lemma3.1} establishes the relation
$$\Ups_{d,e,\pi}(P,R;\phi)\ll P^\eps \int_0^1|F_{d,e}(\alp)^2f(\alp;2Q/e,R)^2|\d\alp ,$$
the conclusion of the lemma follows on noting that $\del_t\ge 0$.
\end{proof}

We also have need of estimates for the mean values
\begin{equation}\label{3.11}
\Lam_{d,e,\pi}^{(m)}(P,R;\phi)=\int_0^1|\Ftil_{d,e}(\alp;\pi)|^{2m}\d\alp \quad 
(m=1,2).
\end{equation}

\begin{lemma}\label{lemma3.3} When $1\le d\le M$, $1\le e\le \min\{Q,Hd^2\}$ and 
$\pi\le R$, one has
$$\Lam_{d,e,\pi}^{(1)}(P,R;\phi)\ll P^{1+\eps}HMe^{-2}\quad \text{and}\quad 
\Lam_{d,e,\pi}^{(2)}(P,R;\phi)\ll P^{2+\eps}H^3M^4e^{-5}.$$
\end{lemma}

\begin{proof} These estimates are given by \cite[equations (3.25) and (3.26)]{Woo1995}.
\end{proof}

Finally, we recall an estimate for the mean value
\begin{equation}\label{3.12}
\Util_s(P,M,R)=\int_0^1\ftil(\alp;P,M,R)^s\d\alp .
\end{equation}

\begin{lemma}\label{lemma3.4}
Suppose that $s>1$ and that $\del_s$ is an associated exponent. Then whenever $P>M$ 
and $R>2$, one has $\Util_s(P,M,R)\ll_s (P/M)^{s/2+\del_s+\eps}$.
\end{lemma}

\begin{proof} This is immediate from \cite[Lemma 3.2]{Woo1995}.
\end{proof}

\section{New associated exponents, I: $4\le s\le 6.5$} We now convert the mean value 
estimates of \S2 into new associated exponents by means of the ideas of 
\cite[\S\S2--4]{Woo1995}. Write
\begin{equation}\label{4.1}
\Ome_{d,e,\pi}(P,R;\phi)=\int_0^1|\Ftil_{d,e}(\alp;\pi)\ftil(\alp;P/(de),M/d,\pi)^{s-2}|
\d\alp ,
\end{equation}
and then put
\begin{equation}\label{4.2}
\calU_s(P,R)=\sum_{1\le d\le D}\sum_{\pi\le R}\sum_{1\le e\le Q}
d^{2-s/2}e^{s/2-1}\Ome_{d,e,\pi}(P,R;\phi).
\end{equation}
The relevant results from \cite{Woo1995} are summarised in the following lemma.

\begin{lemma}\label{lemma4.1} Suppose that $s>4$ and $0<\phi\le \tfrac{1}{3}$. Then 
whenever $\mu_{s-2}$ and $\mu_s$ are permissible exponents, and $1\le D\le P^{1/3}$, 
one has
$$U_s(P,R)\ll P^{\mu_s+\eps}D^{s/2-\mu_s}+MP^{1+\mu_{s-2}+\eps}+
P^{\left(\tfrac{s-3}{s-2}\right)\mu_s+\eps}V_s(P,R),$$
where
$$V_s(P,R)=\left( PM^{s-2}Q^{\mu_{s-2}}+M^{s-3}\calU_s(P,R)\right)^{1/(s-2)}.$$
\end{lemma}

\begin{proof} The desired result follows at once on substituting the conclusion of 
\cite[Lemma 3.3]{Woo1995} into that of \cite[Lemma 2.3]{Woo1995}.
\end{proof}

We are now equipped to announce our new associated exponents.

\begin{lemma}\label{lemma4.2} Suppose that $s\ge 4$ and $0\le \gam\le \tfrac{1}{4}$, 
and let $t$ satisfy
\begin{equation}\label{4.3}
\frac{2s-6+8\gam}{1+2\gam}\le t\le \frac{2s-4}{1+2\gam}.
\end{equation}
Suppose that $\del_{s-2}$ and $\del_t$ are associated exponents, and put
\begin{equation}\label{4.4}
\tet_0=\frac{2s-4-t+2(s-2)\del_t-2t\del_{s-2}}
{6s-12+t-4\gam t+2(s-2)\del_t-2t\del_{s-2}}.
\end{equation}
Then the exponent $\del_s=\del_{s-2}(1-\tet)+\tfrac{1}{2}(s-2)\tet$ is associated, where 
we write $\tet=\max\{0,\min\{\tet_0,\tfrac{1}{3}\}\}$.
\end{lemma}

\begin{proof} We begin by estimating the mean value $\Ome_{d,e,\pi}(P,R;\phi)$. 
Suppose that
$$d\le M,\quad e\le \min\{Q,Hd^2\},\quad \pi\le R\quad \text{and}\quad 0\le \phi\le 
\tfrac{1}{3}.$$
Then on recalling (\ref{3.4}), (\ref{3.11}) and (\ref{3.12}), an application of H\"older's 
inequality to (\ref{4.1}) yields the bound
\begin{align}
\Ome_{d,e,\pi}(P,R;\phi)\le &\, \Ups_{d,e,\pi}(P,R;\phi)^{\gam_1}
\Util_t(P/(de),M/d,\pi)^{\gam_2}\notag \\
&\, \times \Lam_{d,e,\pi}^{(1)}(P,R;\phi)^{\gam_3}
\Lam_{d,e,\pi}^{(2)}(P,R;\phi)^\gam,
\label{4.5}
\end{align}
where
$$\gam_1=\tfrac{1}{4}(2s-4-t-2t\gam),\quad \gam_2=(s-2-2\gam_1)/t\quad 
\text{and}\quad \gam_3=\tfrac{1}{2}-\gam_1-2\gam .$$

\par A few words are in order to confirm that the above is indeed a valid application of 
H\"older's inequality. Observe first that the hypotheses $s>4$ and 
$0\le \gam\le \tfrac{1}{4}$, together with those concerning the value $t$, ensure that
$$2s-6+8\gam\le t(1+2\gam)\le 2s-4,$$
so that
$$0\le \gam_1\le \tfrac{1}{4}\left( (2s-4)-(2s-6+8\gam)\right) =
\tfrac{1}{2}(1-4\gam)\le 1.$$
Hence we deduce that
$$0=\tfrac{1}{2}-\tfrac{1}{2}(1-4\gam)-2\gam\le \gam_3\le \tfrac{1}{2}-2\gam<1.$$
Also, since $s\ge 4$ and $\gam_1\le \tfrac{1}{2}(1-4\gam)$, one finds that
$$\gam_2\ge (s-3+4\gam)/t>0.$$
Moreover, since $t\ge (2s-6+8\gam)/(1+2\gam)$, we have
$$(1+2\gam)(s-2-2\gam_1-t)\le 4-s-2\gam_1-\gam(12+4\gam_1-2s).$$
When $4\le s\le 6$, we therefore deduce that
$$t(1+2\gam)(\gam_2-1)\le 4-s-2\gam_1\le 0,$$
and when $s>6$ we see instead that
$$t(1+2\gam)(\gam_2-1)\le 4-s-2\gam_1+\tfrac{1}{4}(2s-12)\le 1-\tfrac{1}{2}s\le 0.$$
Thus, in all circumstances, one has $0\le \gam_2\le 1$. Finally, the relations
\begin{equation}\label{4.6}
\gam+\gam_1+\gam_2+\gam_3=1,\quad 4\gam+2\gam_1+2\gam_3=1\quad 
\text{and}\quad 2\gam_1+t\gam_2=s-2
\end{equation}
follow by direct computation.\par

By applying Lemmata \ref{lemma3.2}--\ref{lemma3.4}, we deduce from (\ref{4.5}) that
\begin{align*}
\Ome_{d,e,\pi}(P,R;\phi)\ll &\, P^\eps 
\left( d^{4/t}e^{-3-2/t}PMH^{1+2/t}Q^{1+2\del_t/t}\right)^{\gam_1}\\
&\times (PMHe^{-2})^{\gam_3}(P^2M^4H^3e^{-5})^{\gam}
\left( (Q/e)^{t/2+\del_t}\right)^{\gam_2}.
\end{align*}
Thus, by making use of the relations (\ref{4.6}) and
$$t\ge 2,\quad \gam_1\le \tfrac{1}{2},\quad 3\gam_1+
\tfrac{1}{2}t\gam_2+2\gam_3+5\gam\ge \tfrac{1}{2}s,\quad 2\gam_1+t\gam=s-2-
\tfrac{1}{2}t,$$
we deduce that
\begin{equation}\label{4.7}
\Ome_{d,e,\pi}(P,R;\phi)\ll de^{-s/2}P^{1/2+\eps}M^{1/2+2\gam}H^{(s-2)/t}
Q^{s/2-1+(s-2)\del_t/t}.
\end{equation}

\par When $e>Hd^2$, one has $F_{d,e}(\alp)=0$, and hence 
$\Ome_{d,e,\pi}(P,R;\phi)=0$. Thus, on substituting (\ref{4.7}) into (\ref{4.2}), we 
discern that
$$\calU_s(P,R)\ll P^{1/2+\eps}M^{1/2+2\gam}H^{(s-2)/t}
Q^{s/2-1+(s-2)\del_t/t}\Sig_0,$$
where
$$\Sig_0=\sum_{1\le d\le D}\sum_{\pi\le R}\sum_{1\le e\le \min\{Q,Hd^2\}}
d^{3-s/2}e^{-1}.$$
We therefore conclude that
$$\calU_s(P,R)\ll D^2P^{1/2+2\eps}M^{1/2+2\gam}H^{(s-2)/t}
Q^{s/2-1+(s-2)\del_t/t}.$$
In the notation of Lemma \ref{lemma4.1}, therefore, we have
$$V_s(P,R)^{s-2}\ll P^\eps M^{s-3}(\Psi_1+D^2\Psi_2),$$
where
$$\Psi_1=PMQ^{\mu_{s-2}}\quad \text{and}\quad 
\Psi_2=P^{1/2}M^{1/2+2\gam}H^{(s-2)/t}Q^{s/2-1+(s-2)\del_t/t}.$$
On recalling (\ref{2.1}) and the definition of an associated exponent, the equation 
$\Psi_1=\Psi_2$ implicitly determines a linear equation for $\phi$, namely
\begin{align*}
1+&\phi+\left( \tfrac{1}{2}(s-2)+\del_{s-2}\right) (1-\phi)\\
&=\tfrac{1}{2}+\left(\tfrac{1}{2}+2\gam\right)\phi+\Bigl(\frac{s-2}{t}\Bigr)(1-3\phi)+
\Bigl( \tfrac{1}{2}(s-2)+\Bigl(\frac{s-2}{t}\Bigr)\del_t\Bigr) (1-\phi).
\end{align*}
A modicum of computation reveals that this equation has solution $\phi=\tet_0$, where 
$\tet_0$ is given by (\ref{4.4}). Put $D=P^\ome$, where $\ome$ is any sufficiently 
small, but fixed, positive number. Then we may follow the discussion of 
\cite[\S4]{Woo1995} to confirm via Lemma \ref{lemma4.1} that whenever 
$\mu_{s-2}=\tfrac{1}{2}(s-2)+\del_{s-2}$ and $\mu_t=\tfrac{1}{2}t+\del_t$ are 
permissible exponents, then so too is
$$\mu_s=\mu_{s-2}(1-\tet)+1+(s-2)\tet.$$
It follows that the exponent $\del_s=\del_{s-2}(1-\tet)+\tfrac{1}{2}(s-2)\tet $ is 
associated, completing the proof of the lemma.
\end{proof}

We highlight three special cases of Lemma \ref{lemma4.2} for future use.

\begin{corollary}\label{corollary4.3} Suppose that $4<s\le 5$. Then whenever 
$\del_{2s-4}\le 2$ is an associated exponent, so too is $\del_s=\tfrac{1}{2}(s-2)\tet$, 
where
$$\tet=\frac{\del_{2s-4}}{4+\del_{2s-4}}.$$
\end{corollary}

\begin{proof} We take $\gam=0$ and $t=2s-4$, so that $\gam$ and $t$ satisfy 
(\ref{4.3}). It follows from Hua's lemma \cite[Lemma 2.5]{Vau1997} that
$$\int_0^1|f(\alp;Q,R)|^4\d\alp \ll Q^{2+\eps},$$
and hence one may take $\del_u=0$ for $0<u\le 4$. With these choices of $s$, 
$\gam$ and 
$t$, one finds that $\del_{s-2}=0$, and hence (\ref{4.4}) gives
$$\tet_0=\frac{2(s-2)\del_t}{8s-16+2(s-2)\del_t}=\frac{\del_{2s-4}}{4+\del_{2s-4}}.$$
But $0\le \del_{2s-4}\le 2$, and hence $0\le \tet_0\le \tfrac{1}{3}$. The conclusion of the 
corollary is now immediate from Lemma \ref{lemma4.2}.
\end{proof}

\begin{corollary}\label{corollary4.4} Suppose that $5\le s\le 6$. Then whenever 
$\del_6\le \tfrac{3}{2}$ is an associated exponent, so too is 
$\del_s=\tfrac{1}{2}(s-2)\tet$, where
$$\tet=\frac{s-5+(s-2)\del_6}{3s-3+(s-2)\del_6}.$$
\end{corollary}

\begin{proof} We take $\gam=0$ and $t=6$, so that $s$, $\gam$ and $t$ satisfy 
(\ref{4.3}). We again have $\del_u=0$ for $0<u\le 4$, and hence $\del_{s-2}=0$. 
Hence (\ref{4.4}) gives
$$\tet_0=\frac{2s-10+2(s-2)\del_6}{6s-6+2(s-2)\del_6}=
\frac{s-5+(s-2)\del_6}{3s-3+(s-2)\del_6}.$$
But by hypothesis, one has $0\le \del_6\le \tfrac{3}{2}$ and $5\le s\le 6$, and hence
$$0\le \tet_0\le \frac{1+4\del_6}{15+4\del_6}\le \tfrac{1}{3}.$$
The conclusion of the corollary therefore follows from Lemma \ref{lemma4.2}.
\end{proof}

\begin{corollary}\label{corollary4.5} Suppose that $6\le s\le \tfrac{13}{2}$. Then 
whenever $\del_{s-2}\le \del_6\le \tfrac{1}{2}$ is an associated exponent, so too is 
$\del_s=\del_{s-2}(1-\tet)+\tfrac{1}{2}(s-2)\tet$, where
$$\tet=\frac{s-5+(s-2)\del_6-6\del_{s-2}}{33-3s+(s-2)\del_6-6\del_{s-2}}.$$
\end{corollary}

\begin{proof} We take $\gam=\tfrac{1}{2}(s-6)$, so that when $6\le s\le \tfrac{13}{2}$, 
one has $0\le \gam\le \tfrac{1}{4}$, and in addition
$$\frac{2s-6+8\gam}{1+2\gam}=6\quad \text{and}\quad 
\frac{2s-4}{1+2\gam}=2+\frac{6}{s-5}\ge 6.$$
We are therefore entitled to apply Lemma \ref{lemma4.2} with $t=6$, in which case
\begin{align*}
\tet_0&=\frac{2s-10+2(s-2)\del_6-12\del_{s-2}}
{6s-6-12(s-6)+2(s-2)\del_6-12\del_{s-2}}\\
&=\frac{s-5+(s-2)\del_6-6\del_{s-2}}{33-3s+(s-2)\del_6-6\del_{s-2}}.
\end{align*}
By hypothesis, we have $6\le s\le \tfrac{13}{2}$ and $\del_{s-2}\le \del_6\le 
\tfrac{1}{2}$, and hence
$$\tet_0\ge \frac{s-5-2\del_6}{2s+3+(s-2)\del_6-6\del_{s-2}}\ge 
\frac{s-6}{2s+3+(s-2)\del_6-6\del_{s-2}}\ge 0,$$
and
$$\tet_0\le \frac{s-5+\tfrac{9}{2}\del_6}{13-2\del_6}\le \frac{\tfrac{3}{2}+
\tfrac{9}{4}}{12}<\tfrac{1}{3}.$$
The conclusion of the corollary therefore follows from Lemma \ref{lemma4.2}.
\end{proof}

\section{New associated exponents, II: $s=6$ and $6.5<s\le 8$} We turn next to 
methods yielding associated exponents when $s=6$, and when $s>6.5$, beginning with 
one generalising that of \cite[Lemma 2.2]{Woo2000b}.

\begin{lemma}\label{lemma5.1} Let $t$ be a real number with $4<t\le 8$. Then 
whenever $\del_6\le \frac{2}{3}$ and $\del_t\le \frac{1}{6}(t-4)$ are associated 
exponents, then so too is
\begin{equation}\label{5.1}
\del_6^\prime =2\max \left\{ \frac{8-t+8\del_t}{24+t+8\del_t},\frac{\del_6}
{4+\del_6}\right\}.
\end{equation}
Moreover, one has
\begin{equation}\label{5.2}
\int_0^1|F(\alp;P)^2f(\alp;P,R)^4|\d\alp \ll P^{3+\del_6^\prime+\eps}.
\end{equation}
\end{lemma}

\begin{proof} On considering the Diophantine equation underlying (\ref{1.3}), one sees 
that
$$U_6(P,R)\ll \int_0^1|F(\alp;P)^2f(\alp;P,R)^4|\d\alp .$$
Consequently, the confirmation of the estimate (\ref{5.2}) suffices to establish that the 
exponent $\del_6^\prime$ defined in (\ref{5.1}) is associated. We put
$$\phi=\max\left\{\frac{8-t+8\del_t}{24+t+8\del_t},\frac{\del_6}{4+\del_6}\right\}.$$
Our hypotheses concerning $t$, $\del_t$ and $\del_6$ ensure that
$$\frac{8-t+8\del_t}{24+t+8\del_t}\le \frac{8-t+\tfrac{4}{3}(t-4)}{24+t+\tfrac{4}{3}
(t-4)}=\frac{8+t}{56+7t}\le \frac{16}{112}=\frac{1}{7},$$
and
$$\frac{\del_6}{4+\del_6}\le \frac{2/3}{4+2/3}=\frac{1}{7},$$
so that $0\le \phi\le \tfrac{1}{7}$. Recall the definitions (\ref{2.1}) and (\ref{2.2}), and 
define $\grm$ and $\grM$ as in the preamble to Lemma \ref{lemma2.1}. Also, when 
$\grB\subseteq [0,1)$, define
\begin{equation}\label{5.3}
I(\grB)=\int_\grB |F_1(\alp)f(\alp;2Q,R)^4|\d\alp .
\end{equation}
Then \cite[inequality (5.3)]{Woo1995} yields the estimate
\begin{equation}\label{5.4}
\int_0^1|F(\alp;P)^2f(\alp;P,R)^4|\d\alp \ll P^\eps M^3\left(PMQ^2+I([0,1))\right) .
\end{equation}

\par We begin with a discussion of the minor arc contribution $I(\grm)$. By applying 
H\"older's inequality to (\ref{5.3}), one obtains
\begin{equation}\label{5.5}
I(\grm)\ll U_t(2Q,R)^{4/t}\biggl( \int_\grm |F_1(\alp)|^{t/(t-4)}\d\alp\biggr)^{1-4/t}.
\end{equation}
Since we may assume that $\del_t$ is an associated exponent, we have
$$U_t(2Q,R)\ll Q^{t/2+\del_t+\eps}.$$
Also, on recalling that $0\le \phi\le \tfrac{1}{7}$, it follows from 
\cite[inequality (5.4)]{Woo1995} together with the argument of the proof of 
\cite[Lemma 3.7]{Vau1989} that
\begin{align*}
\int_\grm |F_1(\alp)|^{t/(t-4)}\d\alp&\le \Bigl( \sup_{\alp\in \grm}|F_1(\alp)|
\Bigr)^{\tfrac{8-t}{t-4}}\int_0^1|F_1(\alp)|^2\d\alp \\
&\ll P^\eps \left( (PM)^{1/2}H\right)^{\tfrac{8-t}{t-4}}PMH.
\end{align*}
Thus we deduce from (\ref{5.5}) that
\begin{equation}\label{5.6}
I(\grm)\ll P^\eps (PM)^{1/2}H^{4/t}Q^{2+4\del_t/t}.
\end{equation}

\par In order to estimate $I(\grM)$, we have merely to follow the argument leading to 
\cite[equation (2.10)]{Woo2000b}. Thus, again making use of the fact that 
$0\le \phi\le \tfrac{1}{7}$, the estimate preceding \cite[equation (2.10)]{Woo2000b} 
gives
\begin{align}
I(\grM)&\ll P^{1+\eps}HM(PQ^{-2})^{2/3}(Q^5)^{1/3}+P^{1+\eps}HM^{1/2}
(PQ^{-2})^{1/2}(Q^{3+\del_6})^{1/2}\notag \\
&\ll P^{1+\eps}HMQ\left( (PQ^{-1})^{2/3}+(P(QM)^{-1})^{1/2}Q^{\del_6/2}\right) .
\label{5.7}
\end{align}
By combining (\ref{5.6}) and (\ref{5.7}), we obtain an estimate for $I([0,1))$. By 
substituting this into (\ref{5.4}) and recalling (\ref{2.1}), we deduce that
$$\int_0^1|F(\alp;P)^2f(\alp;P,R)^4|\d\alp \ll P^{3+\eps}M^2
(1+\Phi_1+\Phi_2+\Phi_3),$$
where
$$\Phi_1=(PM)^{-1/2}H^{4/t}Q^{4\del_t/t},\quad \Phi_2=M^{-4/3}\quad \text{and}
\quad \Phi_3=M^{-2}Q^{\del_6/2}.$$
In view of (\ref{2.1}), one finds that the respective conditions
$$\phi\ge \frac{8-t+8\del_t}{24+t+8\del_t}\quad \text{and}\quad \phi\ge 
\frac{\del_6}{4+\del_6}$$
ensure that $\Phi_1\le 1$ and $\Phi_3\le 1$. Thus, our choice of $\phi$ ensures that
$$\int_0^1|F(\alp;P)^2f(\alp;P,R)^4|\d\alp \ll P^{3+\eps}M^2=P^{3+2\phi+\eps},$$
confirming the estimate (\ref{5.2}) and completing the proof of the lemma.
\end{proof}

We recall also an estimate for associated exponents $\del_s$ of use when 
$s>\tfrac{13}{2}$. 

\begin{lemma}\label{lemma5.2} Suppose that $s>4$. Then whenever 
$\del_{s-2}\le \tfrac{1}{4}$ and $\del_{4(s-2)/3}\le 1$ are associated exponents, so too 
is $\del_s=\del_{s-2}(1-\tet)+\tfrac{1}{2}(s-2)\tet$, where
$$\tet=\frac{1+3\del_{4(s-2)/3}-4\del_{s-2}}{9+3\del_{4(s-2)/3}-4\del_{s-2}}.$$
\end{lemma}

\begin{proof} This is immediate from \cite[Corollary to Lemma 2]{BBW1995}.
\end{proof}

Finally, we recall a simple consequence of convexity.

\begin{lemma}\label{lemma5.3} Suppose that $s>2$ and $t<s$. Then, whenever 
$\del_{s-t}$ and $\del_{s+t}$ are associated exponents, so too is 
$\del_s=\tfrac{1}{2}(\del_{s+t}+\del_{s-t})$.
\end{lemma}

\begin{proof} This is \cite[Lemma 4.3]{BW2001}.
\end{proof}

\section{The Keil-Zhao device} Lilu Zhao \cite[equation (3.10)]{Zha2014} has observed 
that, in wide generality, one may obtain an estimate of Weyl-type for an exponential sum 
over an arbitrary set, provided this sum inhabits an appropriate mean valuee. The same 
idea is applied also in independent work of Keil \cite[page 608]{Kei2014}. This observation 
is useful in obtaining permissible exponents $\mu_s$ when $s>6$. Before announcing our 
conclusions, we introduce some notation useful in its proof. Write
\begin{equation}\label{6.1}
g(\alp;P,R)=\sum_{\substack{x\in \calA(P,R)\\ x>P/2}}e(\alp x^3)\quad \text{and}\quad 
G(\alp)=\sum_{P/2<x\le P}e(\alp x^3).
\end{equation}

\begin{lemma}\label{lemma6.1} Suppose that $s\ge 6$ and the exponent $\Del_s$ is 
admissible. Suppose also that $\tfrac{1}{16}(8-s)\le \Del_s\le \tfrac{1}{4}$ and 
$u>s+8\Del_s$. Then there exist positive numbers $\eta$ and $c$, depending at most on 
$u$, with the following property. Whenever $P$ is sufficiently large in terms of $\eta$, and 
$\exp\left( c(\log \log P)^2\right)\le R\le P^\eta$, then
\begin{equation}\label{6.2}
\int_0^1|f(\alp;P,R)|^u\d\alp \ll P^{u-3}.
\end{equation}
In particular, the exponent $\mu_w=w-3$ is permissible for $w\ge u$.
\end{lemma}

\begin{proof} We seek to show that whenever $v\ge s+8\Del_s$, then
\begin{equation}\label{6.3}
\int_0^1|f(\alp;P,R)|^v\d\alp \ll P^{v-3+\eps}.
\end{equation}
When $u>v$, the bound (\ref{6.2}) follows from this estimate via 
\cite[Lemma 4.5]{BW2001}. Next, by applying a dyadic dissection, we deduce from 
(\ref{1.1}) and (\ref{6.1}) that
$$f(\alp;P,R)=\sum^\infty_{\substack{j=0\\ 2^j\le \sqrt{P}}}g(\alp;2^{-j}P,R)
+O(\sqrt{P}),$$
whence an application of H\"older's inequality reveals that
\begin{align*}
\int_0^1|f(\alp;P,R)|^v\d\alp &\ll (\log P)^{v-1}\sum^\infty_{\substack{j=0\\ 2^j\le 
\sqrt{P}}} \int_0^1|g(\alp;2^{-j}P,R)|^v\d\alp+P^{v/2}\\
&\ll P^\eps \max_{\sqrt{P}\le X\le P}\int_0^1|g(\alp;X,R)|^v\d\alp +P^{v/2}.
\end{align*}
Consequently, provided we are able to show that 
\begin{equation}\label{6.4}
\int_0^1|g(\alp;P,R)|^v\d\alp\ll P^{v-3+\eps},
\end{equation}
then the bound (\ref{6.3}) follows. Henceforth, we abbreviate $g(\alp;P,R)$ to $g(\alp)$.

\par We establish (\ref{6.4}) via the Hardy-Littlewood method. When $1\le X\le P$, define 
the major arcs $\grM(X)$ to be the union of the intervals
$$\grM(q,a;X)=\{\alp\in[0,1):|q\alp-a|\le XP^{-3}\},$$
with $0\le a\le q\le X$ and $(a,q)=1$. Also, put $\grm(X)=[0,1)\setminus \grM(X)$. Finally, 
write $\grP=\grM(P^{4/5})$, $\grQ=\grM(P^{3/8})$, $\grp=\grm(P^{4/5})$ and 
$\grq=\grm(P^{3/8})$.\par

We begin by observing that, as a consequence of \cite[Corollary 3.2]{BW2001}, one has
$$\int_\grQ |f(\alp;P,R)|^6\d\alp +\int_\grQ|f(\alp;P/2,R)|^6\d\alp \ll P^{3+\eps},$$
so that
$$\int_\grQ |g(\alp)|^6\d\alp \ll P^{3+\eps}.$$
Since $|g(\alp)|=O(P)$, we find that whenever $v\ge 6$, one has
\begin{equation}\label{6.5}
\int_\grQ|g(\alp)|^v\d\alp \ll P^{v-3+\eps}.
\end{equation}

\par Suppose next that $\alp\in \grq$. By Dirichlet's theorem on Diophantine 
approximation, there exist $a\in \dbZ$ and $q\in \dbN$ with $(a,q)=1$, $q\le P^{11/5}$ 
and $|q\alp-a|\le P^{-11/5}$. An application of \cite[Lemma 2.2]{BW2001} in concert 
with \cite[equation (2.1)]{BW2001} delivers the estimate
$$g(\alp)\ll \frac{q^{\eps-1/6}P(\log P)^{5/2+\eps}}{(1+P^3|\alp-a/q|)^{1/3}}
+P^{9/10+\eps}.$$
When $\alp\in \grp$, it follows that $q>P^{4/5}$, and thus $g(\alp)\ll P^{9/10+\eps}$. 
Meanwhile, when $\alp\in \grP\cap \grq$, we have either $q>P^{3/8}$ or 
$|q\alp-a|>P^{-21/8}$, and hence $|g(\alp)|\ll P^{15/16+\eps}$. Consequently, since 
$\grq=\grp\cup (\grP\cap \grq)$, we conclude that
\begin{equation}\label{6.6}
\sup_{\alp\in \grq}|g(\alp)|\ll P^{15/16+\eps}.
\end{equation}

\par We now turn to the main task at hand. Suppose that $s\ge 6$ and that $\Del_s$ 
is an admissible exponent. We consider the mean value
\begin{equation}\label{6.7}
T_0=\int_\grq |g(\alp)|^{s+2}\d\alp .
\end{equation}
By reference to (\ref{6.1}), an application of Cauchy's inequality shows that
\begin{equation}\label{6.8}
T_0=\sum_{\substack{x\in \calA(P,R)\\ x>P/2}}\sum_{\substack{y\in \calA(P,R)\\ 
y>P/2}}\int_\grq |g(\alp)|^se(\alp(x^3-y^3))\d\alp \le PT_1^{1/2},
\end{equation}
where
$$T_1=\sum_{\substack{P/2<x,y\le P\\ x,y\in \calA(P,R)}}\biggl| \int_\grq 
|g(\alp)|^se(\alp(x^3-y^3))\d\alp \biggr|^2.$$
We bound $T_1$ above by removing the condition $x,y\in \calA(P,R)$, obtaining
$$T_1\le \sum_{P/2<x,y\le P}\int_\grq \int_\grq |g(\alp)g(\bet)|^se\left( (\alp-\bet)
(x^3-y^3)\right) \d\alp \d\bet .$$
Thus, again recalling (\ref{6.1}), we deduce by means of (\ref{6.8}) that
\begin{equation}\label{6.9}
T_0^2\le P^2\int_\grq \int_\grq |g(\alp)g(\bet)|^s|G(\alp-\bet)|^2\d\alp \d\bet .
\end{equation}

\par We analyse the mean value on the right hand side of (\ref{6.9}) by means of the 
Hardy-Littlewood method. Let $\grN=\grM(P^{3/4})$ and $\grn=\grm(P^{3/4})$. Denote 
by $\kap(q)$ the multiplicative function defined on prime powers by taking
$$\kap(p^{3l})=p^{-l},\quad \kap(p^{3l+1})=3p^{-l-1/2},\quad 
\kap(p^{3l+2})=p^{-l-1}\quad (l\ge 0).$$
Also, define the function $\Ups(\gam)$ for $\gam\in \grN$ by taking
\begin{equation}\label{6.10}
\Ups(\gam)=\kap(q)^2(1+P^3|\gam-a/q|)^{-1},
\end{equation}
when $\gam\in \grM(q,a;P^{3/4})\subseteq \grN$, and put $\Ups(\gam)=0$ when 
$\gam \in \grn$. Then it follows from \cite[Lemma 2.1]{KW2001} that 
$G(\gam)^2\ll P^2\Ups(\gam)+P^{3/2+\eps}$. Substituting this estimate into (\ref{6.9}), 
we deduce that
\begin{equation}\label{6.11}
T_0^2\ll P^{7/2+\eps}\biggl( \int_0^1 |g(\alp)|^s\d\alp \biggr)^2+P^4T_2,
\end{equation}
where
$$T_2=\int_\grq \int_\grq \Ups(\alp-\bet)|g(\alp)g(\bet)|^s\d\alp \d\bet .$$

\par By applying the trivial inequality $|z_1\cdots z_n|\le |z_1|^n+\ldots +|z_n|^n$, we 
find that
$$|g(\alp)g(\bet)|^s\ll |g(\alp)g(\bet)^{s-1}|^2+|g(\bet)g(\alp)^{s-1}|^2.$$
Hence, by symmetry, we obtain the estimate
$$T_2\ll \Bigl( \sup_{\bet\in \grq}|g(\bet)|\Bigr)^{s-4}\int_\grq \int_\grq 
\Ups(\alp-\bet)|g(\bet)^{s+2}g(\alp)^2|\d\alp \d\bet .$$
By invoking (\ref{6.6}), we thus deduce that
\begin{equation}\label{6.12}
T_2\ll (P^{15/16+\eps})^{s-4}\int_\grq |g(\bet)|^{s+2}\int_0^1\Ups(\alp-\bet)
|g(\alp)|^2\d\alp \d\bet .
\end{equation}

\par On recalling the definitions (\ref{6.1}) and (\ref{6.10}), we discern that
\begin{align*}
\int_0^1\Ups(\alp-\bet)|g(\alp)|^2\d\alp &=\int_\grN \Ups(\gam)
|g(\gam+\bet)|^2\d\gam \le \sum_{1\le q\le P^{3/4}}\kap(q)^2\Lam(q),
\end{align*}
where
$$\Lam(q)=\sum^q_{\substack{a=1\\ (a,q)=1}}\int_{-P^{-9/4}}^{P^{-9/4}}
(1+P^3|\tet|)^{-1}\biggl| \sum_{\substack{x\in \calA(P,R)\\ x>P/2}}
e(x^3(\bet+\tet+a/q))\biggr|^2\d\tet .$$
Let $c_q(n)$ be Ramanujan's sum, which we define by
$$c_q(n)=\sum^q_{\substack{a=1\\ (a,q)=1}}e(an/q).$$
Then it follows that
$$\sum^q_{\substack{a=1\\ (a,q)=1}}\biggl| \sum_{\substack{x\in \calA(P,R)\\ 
x>P/2}}e(x^3(\bet+\tet+a/q))\biggr|^2=\sum_{\substack{P/2<x,y\le P\\ 
x,y\in \calA(P,R)}}c_q(x^3-y^3)e((\bet+\tet)(x^3-y^3)).$$
Thus, the well-known estimate $|c_q(n)|\le (q,n)$ yields the bound
$$\Lam(q)\le \sum_{1\le x,y\le P}(q,x^3-y^3)\int_{-P^{-9/4}}^{P^{-9/4}}
(1+P^3|\tet|)^{-1}\d\tet,$$
and consequently
$$\int_0^1\Ups(\alp-\bet)|g(\alp)|^2\d\alp \ll P^{-3}\log (2P)\sum_{1\le q\le  P^{3/4}}
\kap(q)^2\sum_{1\le x,y\le P}(q,x^3-y^3).$$
From here, the treatment following \cite[equation (3.2)]{BW2001} delivers the upper 
bound
\begin{equation}\label{6.13}
\int_0^1\Ups(\alp-\bet)|g(\alp)|^2\d\alp \ll P^{\eps-1}.
\end{equation}

\par Next, substituting (\ref{6.13}) into (\ref{6.12}), we infer that
$$T_2\ll P^{\eps-1}(P^{15/16})^{s-4}\int_\grq|g(\bet)|^{s+2}\d\bet.$$
In view of (\ref{6.7}) and (\ref{6.11}), the hypothesis that $\Del_s$ is admissible yields
$$T_0^2\ll P^{7/2+\eps}\left( P^{s-3+\Del_s}\right)^2+P^{3+\eps}\left( 
P^{15/16}\right)^{s-4}T_0,$$
whence
$$T_0\ll P^{s-1+\eps}\left( P^{\Del_s-1/4}+P^{-(s-4)/16}\right) .$$
On recalling (\ref{6.7}), application of H\"older's inequality and the trivial estimate 
$|g(\alp)|\le P$ delivers the upper bound
\begin{align*}
\int_\grq |g(\alp)|^v\d\alp&\le P^{v-(s+8\Del_s)}T_0^{4\Del_s}
\biggl( \int_0^1|g(\alp)|^s\d\alp \biggr)^{1-4\Del_s}\\
&\ll  P^{v-s-8\Del_s+\eps}\left( P^{s}\left( 
P^{\Del_s-5/4}+P^{-(s+12)/16}\right) \right)^{4\Del_s}\left( 
P^{s-3+\Del_s}\right)^{1-4\Del_s}.
\end{align*}
Thus we deduce that whenever $\Del_s\ge \tfrac{1}{16}(8-s)$, then
$$\int_\grq |g(\alp)|^v\d\alp \ll P^{v-3+\eps}\left(1+P^{-\Del_s+(8-s)/16}\right)\ll 
P^{v-3+\eps}.$$
But the latter condition on $s$ is assured by the hypotheses of the lemma, and thus 
we conclude via (\ref{6.5}) that
$$\int_0^1|g(\alp)|^v\d\alp =\int_\grQ |g(\alp)|^v\d\alp +\int_\grq|g(\alp)|^v
\d\alp \ll P^{v-3+\eps}.$$
This confirms the estimate (\ref{6.4}), and the conclusion of the lemma follows.
\end{proof}

\section{Computations} We now address the problem of how to implement the 
computation of associated exponents $\del_s$ for $4\le s\le 8$. Let $h$ be a small 
positive number that we view as a step size, and put $J=\lceil 16/h\rceil$. It is convenient in 
what follows to assume that $1/h\in \dbN$. We begin with an array of known associated 
exponents $\del_{jh}$ $(0\le j\le J)$. Thus, we have the associated exponents 
$\del_4=0$ and $\del_s=\tfrac{1}{2}s-3$ $(s\ge 8)$ which follow from Hua's lemma (see 
\cite[Lemma 2.5]{Vau1997}). Making use also of the associated exponent 
$\del_6=\tfrac{1}{4}$ due to Vaughan \cite[Theorem 4.4]{Vau1989}, one may apply 
convexity to deliver the associated exponents
$$\del_s=\max\left\{ 0,\tfrac{1}{8}(s-4),\tfrac{3}{8}s-2,\tfrac{1}{2}s-3\right\} .$$
For the interesting values of $j$ with $4<jh<8$, one may now calculate new associated 
exponents $\del_{jh}$ by means of Lemmata \ref{lemma4.2}, 
\ref{lemma5.1}-\ref{lemma5.3} and \ref{lemma6.1}. Here, we note that associated 
exponents $\del_s$ are related to admissible exponents $\Del_s$ by means of the 
relation $\del_s=\tfrac{1}{2}s-3+\Del_s$. Should any of these new associated exponents 
be superior to the old ones, then they may be substituted into the array of values 
$\del_{jh}$. By iterating this process for $4/h<j<8/h$, one derives new associated 
exponents converging to some set of limiting values.\par

We summarise the formulae delivered by the above-cited lemmata as follows.\vskip.1cm

\noindent (i) \emph{Method $A_s(t,\gam)$.} We apply Lemma \ref{lemma4.2} for 
$\gam=lh$ and $t=mh$ with $0\le l\le (4h)^{-1}$ and
\begin{equation}\label{7.1}
\frac{2jh-6+8lh}{1+2lh}\le mh\le \frac{2jh-4}{1+2lh}.
\end{equation}
Thus one finds that the exponent $\del_{jh}^\prime$ is associated, where
\begin{equation}\label{7.2}
\del_{jh}^\prime=\del_{jh-2}(1-\tet)+\tfrac{1}{2}(jh-2)\tet ,
\end{equation}
in which $\tet=\max\left\{ 0,\min\left\{ \tet_0,\tfrac{1}{3}\right\} \right\}$, and
$$\tet_0=\frac{2jh-4-mh+2(jh-2)\del_{mh}-2mh\del_{jh-2}}{6jh-12+mh-4(lh)(mh)+
2(jh-2)\del_{mh}-2mh\del_{jh-2}}.$$
\vskip.1cm

\noindent (ii) \emph{Method $B_6(t)$.} We apply Lemma \ref{lemma5.1} for $t=mh$ 
with $4<mh\le 8$. Thus, when $\del_{mh}\le \tfrac{1}{6}(mh-4)$, we find that the 
exponent $\del_6^\prime$ is associated, where
$$\del_6^\prime =2\max \left\{ \frac{8-mh+8\del_{mh}}{24+mh+8\del_{mh}},
\frac{\del_6}{4+\del_6}\right\} .$$
\vskip.1cm

\noindent (iii) \emph{Method $C_s$.} First, if $i$ is the integer for which 
$\tfrac{4}{3}(j-2/h)\in (i,i+1]$, then convexity provides the associated exponent
$$\del_{4(jh-2)/3}=\left(i+1-\tfrac{4}{3}(j-2/h)\right)\del_{ih}+
\left(\tfrac{4}{3}(j-2/h)-i\right)\del_{(i+1)h}.$$
Next, Lemma \ref{lemma5.2} shows the exponent $\del_{jh}^\prime$ given by 
(\ref{7.2}) to be associated, where
$$\tet_0=\frac{1+3\del_{4(jh-2)/3}-4\del_{jh-2}}{9+3\del_{4(jh-2)/3}-4\del_{jh-2}}.$$
\vskip.1cm

\noindent (iv) \emph{Process $L_s(t)$.} We apply Lemma \ref{lemma5.3} for $t=mh$ 
with $1\le m\le 1/h$. Thus one finds that the exponent $\del_{jh}^\prime$ is associated, 
where $\del_{jh}^\prime =\tfrac{1}{2}(\del_{(j+m)h}+\del_{(j-m)h})$.
\vskip.1cm

\noindent (v) \emph{Process $W_s$.} We apply Lemma \ref{lemma6.1}. Thus one finds that 
$\del_{jh}^\prime=\tfrac{1}{2}jh-3$ is an associated exponent whenever $\del_{jh-mh}$ 
is associated and satisfies
$$3-\tfrac{1}{2}(j-m)h+\del_{jh-mh}<\tfrac{1}{8}mh.$$
\vskip.1cm

We wrote a straightforward computer program to implement this iterative process. Our 
language of choice was the QB64 implementation of QuickBasic, running on a Windows 
Surface Pro3 in Windows 8.1 (Intel Core i3 processor at 1.5 GHz). All parameters were 
stored using double precision variables. The most time consuming method to apply is 
process $A_s(t,\gam)$, since there are many possible choices for $t=mh$ and $\gam=lh$ 
to test. It is apparent that $\gam$ should be chosen as small as possible consistent with 
the constraint (\ref{7.1}). However, applying process $A_s(t,\gam)$ for each eligible 
value of $s=jh$ $(4<s<8)$ nonetheless has running time with order of growth $h^{-2}$. 
This limited our computation, in the first instance, to a step size of $h\ge 10^{-4}$.\par

Having experimented with this iteration, it becomes apparent that certain of the processes 
dominate the others for different values of $s$. By refining the program to select dominant 
processes for different ranges of $s$, the running time is vastly improved to order of growth 
$h^{-1}$. Note that the array size limit effective for QB64 on the platform employed was 
at least $2\times 10^8$. Thus, final computations with step size $h=10^{-6}$ were 
feasible for $4<s\le 6.5$, and step size $h=10^{-5}$ throughout $4<s\le 8$, this being 
limited only by running-time considerations rather than memory limitations. We summarise 
below the parameters associated with these dominant processes.\vskip.1cm

\noindent (i) $4<s\le 5$. Process $A_s(2s-4,0)$, so that $\del_s^\prime$ is determined 
according to Corollary \ref{corollary4.3}. Thus $\del_{jh}^\prime$ is given by 
(\ref{7.2}) with
$$\tet_0=\frac{\del_{2jh-4}}{4+\del_{2jh-4}}.$$
\vskip.1cm

\noindent (ii) $5<s\le 5.6462$. Process $A_s(6,0)$, so that $\del_s^\prime$ is determined 
according to Corollary \ref{corollary4.4}. Thus $\del_{jh}^\prime$ is given by (\ref{7.2}) 
with
$$\tet_0=\frac{jh-5+(jh-2)\del_6}{3jh-3+(jh-2)\del_6}.$$
\vskip.1cm

\noindent (iii) $5.6462<s<6$. Process $L_s(t)$, linear interpolation between 
$\del_{5.6462}$ and $\del_6$.\vskip.1cm

\noindent (iv) $s=6$. Process $B_6(5.392938)$.\vskip.1cm

\noindent (v) $6<s\le 6.081$. Process $L_s(t)$, linear interpolation between 
$\del_6$ and $\del_{6.081}$.\vskip.1cm

\noindent (vi) $6.081<s\le 6.3395$. Process $A_s(6,\tfrac{1}{2}(s-6))$, so that 
$\del_s^\prime$ is determined according to Corollary \ref{corollary4.5}. Thus 
$\del_{jh}^\prime$ is given by (\ref{7.2}) with
$$\tet_0=\frac{jh-5+(jh-2)\del_6-6\del_{jh-2}}{33-3jh+(jh-2)\del_6-6\del_{jh-2}}.$$
\vskip.1cm

\noindent (vii) $6.3395<s\le 6.5$. Process $L_s(t)$, linear interpolation between 
$\del_{6.3395}$ and $\del_{6.5}$.\vskip.1cm

\noindent (viii) $6.5<s\le 7.06$. Processes $C_s$ and $L_s(t)$.\vskip.1cm

\noindent (ix) $7.06<s<8$. Processes $W_s$ and $L_s(t)$.\vskip.1cm

Some additional discussion seems warranted concerning the robustness of these 
computations. The first point to make is that, while the above restricted iteration may 
not be guaranteed to deliver optimal estimates, the exponents that it delivers will at least 
be legitimate associated exponents. Thus the exponents presented in Table 1 in the 
introduction may be considered upper bounds for optimal associated exponents. In this 
context, it is worth noting that we experimented with adjustments to the step size $h$, 
and found no improvement in the first $8$ digits of the decimal expansions of the 
computed values of $\del_s$, even when $h$ varied from $10^{-4}$ to $10^{-6}$.\par

The second point concerns the stability of the iteration. There is a potential danger in 
iterations involving large numbers of cycles that round-off errors may accumulate, leading 
to substantial cumulative errors and even to unstable iterative processes. In our 
computations, we exercised some caution concerning this issue by artificially inflating the 
newly computed associated exponents by adding a small positive quantity $\tau$ at the 
end of each iteration. Thus, with $\tau=10^{-9}$, we replaced the newly computed 
associated exponent $\del_s$ by $\del_s+\tau$. This has the effect of slightly weakening 
our exponents, though round-off errors (which in double-precision arithmetic are very 
much smaller) are swamped by this cushion of numerical security. This device has the 
effect of permitting some control on the number of decimal digits reliably computed.\par

We now interpret these computations in the context of the 
conclusions presented in the introduction. First, Theorem \ref{theorem1.2} follows from 
the computed associated exponent $\del_t=0.14963020$ for $t=5.392938$ that follows 
from the computations underlying Table 1 via convexity, and the upper bound (\ref{5.2}) 
of Lemma \ref{lemma5.1}. Next, the exponent $\Del_{7.1}=0.06131437$ is admissible, 
according to Theorem \ref{theorem1.5} and the associated Table 1. Then it follows from 
Lemma \ref{lemma6.1} that
$$\int_0^1|f(\alp;P,R)|^u\d\alp \ll P^{u-3}$$
whenever $u>7.1+8\Del_{7.1}=7.59051\ldots $. This establishes Theorem \ref{theorem1.4}. Finally, the proof 
of Theorem \ref{theorem1.1} is a standard consequence of Theorem \ref{theorem1.2}, 
following an application of Cauchy's inequality. The proof of 
\cite[Theorem 1.1]{Woo2000b} to be found in the final phases of \cite[\S2]{Woo2000b} 
shows, for example, that whenever $\del_6$ is an associated exponent, then 
$N(X)\gg X^{1-\del_6/3-\eps}$. The conclusion of Theorem \ref{theorem1.1} therefore follows on making use of the 
associated exponent $\del_6=0.24871567$. Note also that Theorem \ref{theorem1.5} for 
$s=4$ follows from \cite{Hoo1963}.

\bibliographystyle{amsbracket}
\providecommand{\bysame}{\leavevmode\hbox to3em{\hrulefill}\thinspace}

\end{document}